\documentclass[12pt]{article}
\usepackage{amsmath,amsthm,amssymb}
\usepackage{mathabx,epsfig}

\usepackage[margin=1in]{geometry}
\usepackage{amsmath,amscd}
\usepackage[all]{xy}
\newtheorem{theorem}{Theorem}
\newtheorem{conjecture}[theorem]{Conjecture}
\newtheorem{corollary}{Corollary}[theorem]
\newtheorem{lemma}{Lemma}[theorem]
\newtheorem{proposition}{Proposition}[theorem]
\usepackage[small,compact]{titlesec}
\usepackage{times}

\makeatletter

\newcommand{\Rmnum}[1]{\expandafter\@slowromancap\romannumeral #1@}
\makeatother

 \newcommand{\Jac}{\operatorname{Jac}}
 \newcommand{\Aut}{\operatorname{Aut}}
\newcommand{\Sel}{\operatorname{Sel}}
\newcommand{\Gal}{\operatorname{Gal}}
\newcommand{\rank}{\operatorname{rank}}
\newcommand{\dimension}{\operatorname{dim}}
\newcommand{\disc}{\operatorname{disc}}

\theoremstyle{remark}
\newtheorem{definition}{Definition}[theorem]
\newtheorem{remark}{Remark}[theorem]

\begin{document}
\title{Points on Elliptic Curves Parametrizing Dynamical Galois Groups }
\author{Wade Hindes\\
Department of Mathematics, Brown University\\
Providence, RI 02912\\
E-mail: whindes@math.brown.edu}
\date{\today}
\maketitle 
\renewcommand{\thefootnote}{}

\footnote{2010 \emph{Mathematics Subject Classification}: Primary 14G05; Secondary 37P55Ê.}

\footnote{\emph{Key words and phrases}: Rational Points on Curves, Arithmetic Dynamics, Galois Theory.}

\renewcommand{\thefootnote}{\arabic{footnote}}
\setcounter{footnote}{0}

\begin{abstract} We show how rational points on certain varieties parametrize phenomena arising in the  Galois theory of iterates of quadratic polynomials. As an example, we characterize completely the set of  quadratic polynomials $x^2+c$ whose third iterate has a ``small" Galois group by determining the rational points on some elliptic curves. It follows as a corollary that the only such integer value with this property is $c=3$, answering a question of Rafe Jones. Furthermore, using a result of Granville's on the rational points on quadratic twists of a hyperelliptic curve, we indicate how the ABC conjecture implies a finite index result, suggesting a geometric interpretation of this problem.       
\end{abstract} 
\begin{section}{Introduction: A Geometric Interpretation of Iterated Galois Behavior}

The study of the Galois behavior of iterates of rational polynomials, begun by Odoni \cite{Odoni}, provides a wealth of interesting unsolved problems linking arithmetic, geometry, and dynamics. Even in the most studied and basic case, the quadratic polynomial, much remains a mystery. We begin by describing the generic situation. Suppose that $f\in\mathbb{Q}[x]$ is a polynomial of degree $d$ whose iterates are separable (the polynomials obtained from successive composition of $f$ have distinct roots in an algebraic closure). If $T_n$ denotes the set of roots of $f, f^2,\dots ,f^n$ together with $0$, then $T_n$ carries a natural $d$-ary rooted tree structure: $\alpha,\beta\in T_n$ share an edge if and only if $f(\alpha) =\beta$. As $f$ is a polynomial with rational coefficients, the Galois group of $f^n$, which we denote by $\Gal(f^n)$, acts via graph automorphisms on $T_n$. Such a framework provides an arboreal representation, $\Gal(f^n)\hookrightarrow \Aut(T_n)$, and we can ask about the size of the image. Suppose that $f$ is a critically infinite quadratic polynomial, that is, the orbit of the unique critical point of $f$ is infinite. It has been conjectured that for such polynomials, the image of the inverse limit, $G(f):=  \varprojlim\Gal(f^n)$, is of finite index in the automorphism group of the full preimage tree, $\Aut(T)$. This is an analog of Serre's result for the Galois action on the prime-powered torsion points of a non CM elliptic curve; see \cite{B-J} for a more complete description.

For integer values $c$, Stoll has given congruence relations which ensure that the Galois groups of iterates of $f_c(x)=x^2+c$ are as large as possible \cite{Stoll-Galois}. However, much is unknown as to the behavior of integer values not meeting these criteria, not to mention the more general setting of rational $c$ (for instance $c=3$ and $-\frac{2}{3}$).  In fact, the state of the art seems to be to analyze the prime divisibility of the critical orbit, $\{f(0),f^2(0), f^3(0)\dots \}$, from which one may be lucky enough to force the Galois groups to be as large a possible \cite{B-J}. 
\begin{remark}	Assuming the ABC conjecture, Gratton, Nguyen, and Tucker have deduced the existence of primitive prime divisors in the critical sequence of rational quadratic maps \cite{PrimitiveDivisors}. 
\end{remark} 
Note that Hilbert's irreducibility theorem implies that $\Gal(f_c^n)\cong\Aut(T_n)$ outside of a thin set of $c$'s, and it is precisely this thin set that we wish to characterize. To do this, it suffices to understand those polynomials $f$ for which $\Gal(f^n)$ is smaller than expected for the first time at level $n$, leading to the following definition: 
\begin{definition} Suppose $n\geq2$. If $f$ is a quadratic polynomial such that $\Gal(f^{n-1})\cong\Aut(T_{n-1})$, yet $\Gal(f^n)\ncong\Aut(T_n)$, then we say that $f$ has a \textit{small} $n$-th iterate. 
\end{definition} 
  
Let \[f_{\gamma,c}(x):=(x-\gamma)^2+c\] be a monic quadratic polynomial over the rational numbers (this seems to be the appropriate parametrized family of quadratic polynomials for studying the Galois theory of iterates, since the critical orbit encodes much of the relevant information: irreducibility, discriminant, etc. \cite{Jones1}). In this paper, we will usually view the critical point as given and study the Galois behavior of $f_{\gamma,c}$ as $c$ varries. With this is mind, we also have the following notation. 
 \begin{definition} For $\gamma\in\mathbb{Q}$, let \[S_{\gamma}^{(n)}:=\{c\in\mathbb{Q}|\,f_{\gamma,c}\,\text{has a small $n$-th iterate}\},\] be the set of rational $c$'s giving rise to a polynomial with small $n$-th iterate.  
\end{definition} 
In section $4$, we will show that $S_\gamma^{(3)}$ is infinite for all but finitely many $\gamma$, utilizing the theory of elliptic surfaces. However, we will first compute two complete examples, $S_0^{(3)}$ and $S_1^{(3)}$, in order to illustrate sensitivity on the critical point $\gamma$. 

To begin our study, fix a critical point $\gamma\in\mathbb{Q}$ and consider $f_{\gamma,t}=f_t$ with $t$ an indeterminate. Define a set of polynomials $\mathcal{G}_m$ in $\mathbb{Q}[t]$ recursively by \[\mathcal{G}_1=\{-t\}\:\; \text{and}\:\;\mathcal{G}_m=\mathcal{G}_{m-1}\cup\{f_t^m(\gamma)g: g\in\mathcal{G}_{m-1}\}\cup\{f_t^m(\gamma)\}\;\text{for}\;m\geq2.\]   

Suppose $c\in\mathbb{Q}$ is such that $\Gal(f_{c}^m)\cong\Aut(T_m)$, the generic situation. Then the collection of fields $\mathbb{Q}(\sqrt{g(c)})$ for $g\in\mathcal{G}_m$ comprise all quadratic subfields of $K_{m,c}$, a splitting field for $f_c^m$.  

\begin{remark}To see that this list of quadratic subfields is exhaustive in the generic setting, note that there are at most $2^m-1$ such subfields, coming from our embedding of the Galois group of $K_{m,c}$ into the $m$-fold wreath product of the cyclic group of order $2$. On the other hand, one can show inductively that $K_{m,c}$ is as large as possible if and only if the elements of the critical orbit, and their products with each other, form distinct quadratic extensions of $\mathbb{Q}$; see Theorem $1$ in \cite{Stoll-Galois} or Lemma $2.2$ in \cite{Jones2}.  
\end{remark}

We will now define the curves which parametrize certain Galois phenomena arising in the dynamics of quadratic polynomials. For a given $\gamma$, set \[C_{\gamma,n}=C_n:=\{(t,y): y^2=f_t^n(\gamma)=( ( (t-\gamma)^2+t)-\gamma)^2+t)\dots -\gamma)^2+t)\},\] 
an affine curve. 
We also define 
\[V_{n,\gamma}=V_n:=C_n\cup\bigcup_{g\in\mathcal{G}_{n-1}} \{(t,y): g(t)y^2=f_t^n(\gamma)\}.\]

If all iterates of $f_c$ are irreducible and $c\in S_\gamma^{(n)}$, then lemma 3.2 of \cite{Jones2} implies:  \[|\Gal(K_{n,c}/K_{n-1,c})|\neq2^{2^{n-1}}\text{and}\quad(c,y)\in V_{n,\gamma}(\mathbb{Q}),\]
which is to say, $\Gal(f_c^n)$ is not maximal for the first time at the $n$-th stage if and only if $(c,y)$ gives a rational point on some curve(s). Note that $g(t)y^2=f_t^n(\gamma)$ maps into the potentially singular hyperelliptic curve $y^2=g(t)\cdot f_t^n(\gamma)$. As we will see later, one can implement various algorithms with these hyperelliptic curves, which facilitate computations.

Now, if $\pi: V_n\rightarrow\mathbb{Q}$ is the map that sends $(t,y)\rightarrow t$, then we make the following conjecture.

\begin{conjecture} For all $\gamma\in\mathbb{Q}$, if $c\in\mathbb{Q}$ is such that $f_c$ is critically infinite and every iterate of $f_c$ is irreducible, then \[\{n|c\in\pi(V_n(\mathbb{Q}))\}\] is finite. In particular, if $\gamma=0$ and $c$ is an integer not equal to $-2$ with $-c$ not a square, then the above set is finite.    
\end{conjecture} 
\begin{remark}Note that Conjecture $1$ implies that for every $\gamma$ and $c$ there exists an $n(\gamma,c)=n(c)$, such that $c\notin S_{\gamma}^{(m)}$ for $m\geq n(c)$, a slightly weaker statement than the one presuming $\Gal(f)$ has finite index in $\Aut(T)$. In words, if one wishes to establish that $\Gal(f^n_{\gamma,c})$ is as large as possible for every $n$, one need only check it to some bounded stage. 
\end{remark}   
Forcing $f$ to satisfy the hypothesis of Conjecture $1$ is relatively easy; see \cite{Jones2}. Hence a first step to prove such a result would be to to understand the varieties $V_n$ for a fixed $\gamma$. Are the $V_n$ nonsingular? If singular, what do their normalizations look like? Note that $V_n(\mathbb{Q})$ will be finite for $n$ large enough \cite{Falt}. After normalizing, can one say anything about the simple factors or ranks of the corresponding Jacobians? How much bigger is $V_n(\mathbb{Q})$ than the points coming from small $n$-th iterates? 

We answer these questions for the first nontrivial case $\gamma=0$ and $n=3$ in Theorem $3$ of section $3$, which states that $V_{3}(\mathbb{Q})=E_1(\mathbb{Q})\cup E_2(\mathbb{Q})$, and each $E_i$ is an elliptic curve whose rational points have rank one. Moreover, the points on $V_{3}(\mathbb{Q})$ not coming from those in $S_0^{(3)}$ have $t$-coordinate either $0$ or $-2$.   
\end{section}

Before establishing this result, we show how the theory of quadratic twists of hyperelliptic curves relates to the dynamics of quadratic polynomials. 
\begin{section}{Twists of Hyperelliptic Curves and the Finite Index Conjecture} 
As a first illustration of using curves to study the dynamics of quadratic polynomials, we show how one can use a theorem of Granville's on the rational points on quadratic twists of a hyperelliptic curve, assuming the ABC conjecture over $\mathbb{Q}$, to prove a finite index result in the case $\gamma=0$ and $c$ is an integer. 
\begin{theorem} Assume the ABC conjecture over $\mathbb{Q}$, and suppose that $c$ is an integer such that $f_c(x)=x^2+c$\; is critically infinite and every iterate of $f_c$ irreducible. Then for all but finitely many $n$, we have that $|\Gal(K_{n,c}/K_{n-1,c})|=2^{2^{n-1}}$. In particular, the full arboreal representation of $\Gal(f)$ inside $\Aut(T)$ has finite index.  
\end{theorem} 
\begin{remark} If $c$ is an integer not equal to $-2$ such that $-c$ is not a square, then $c$ will satisfy the hypotheses of Theorem $2$; see \cite{Jones2}. 
\end{remark}
\begin{proof}We fix $c$ and suppress it in all notation. First note that the adjusted critical orbit $\{-c,f^2(0),f^3(0)\dots\}$ of $f$ can contain at most finitely many squares, for otherwise we would obtain infinitely many integer points $\{(f^{n_k-2}(0),y_{n_k})\}_{k=1}^\infty$ on the smooth hyperelliptic curve $y^2=f^{2}(x)$, and moreover this set of points is unbounded ($f$ is critically infinite). This contradicts Siegel's theorem on the finiteness of integral points on curves of genus at least one \cite{Silv}. 
\begin{remark} For this particular family of quadratic polynomials, one can show that the critical orbit in fact does not contain any squares, see Corollary 1.3 of \cite{Stoll-Galois}. However, this sort of argument generalizes.  
\end{remark}

Therefore, we may assume that if $|\Gal(K_{n}/K_{n-1})|\neq2^{2^{n-1}}$, then for some $y\in\mathbb{Q}$,
\[dy^2=f^n(0), \; \text{with}\;\;\mathbb{Q}(\sqrt{d})\subset K_{n-1};\] see \cite{Stoll-Galois} Lemma 1.6. It is our goal to show that $n$ is bounded. Note that the curve \[C:= \{dy^2=f^2(x)\}\; \text{has a rational point} \; (f^{n-2}(0),y),\] and $C$ is smooth since all iterates of $f$ are assumed to be separable. Moreover, $d=\prod_ip_i$, where the $p_i$'s are distinct primes dividing $2\cdot\prod_{j=1}^{n-1}f^j(0)$. To see this latter fact, we use a formula for the discriminant $\Delta_m$ of $f^m$, \[\Delta_m=\pm\Delta_{m-1}^2\cdot2^{2^m}\cdot f^m(0),\] given in Lemma 2.6 of \cite{Jones2}. It follows that the only rational primes which ramify in $K_{n-1}$ are the primes dividing  $2\cdot\prod_{j=1}^{n-1}f^j(0)$. Since the primes which divide $d$ must ramify in $K_{n-1}$, we obtain the desired description of the $p_i$.  

Now we apply a theorem of Granville's on the rational points on twists of a hyperelliptic curve,  which assumes the ABC conjecture, to $C$ (Theorem 1 of \cite{Twists}). This yields  
\begin{equation}\label{Twist} |f^{n-2}(0)| \ll |d|^{\frac{1}{2+ o(1)}}.\end{equation}

We claim that each $p_i$ must also divide $f^n(0)$. To see this, write $v=a/b$, for coprime integers $a$ and $b$, so that $p_1p_2\dots p_ta^2=b^2f^n(0)$. Hence each $p_i$ divides $b$ or $f^n(0)$. If $p_i$ divides $b$, then the fact that these primes are assumed to be distinct also implies that $p_i$ must divide $a$. It follows that $v$ must be an integer, and that $p_i$ divides $f^n(0)$ as claimed. 

However, we also have that the odd $p_i$ divide $f^j(0)$ for some $j\leq n-1$, so that such $p_i$ must divide $f^{n-j}(0)$. Therefore, we can choose $j\leq \lfloor\frac{n}{2}\rceil$ for every odd $p_i$, where $\lfloor x \rceil$ denotes the nearest integer function. Since $d$ is assumed to be square free, we obtain  
\begin{equation}\label{Twist} |f^{n-2}(0)| \ll |d|^{\frac{1}{2+ o(1)}}\leq|2\cdot f(0)\cdot f^2(0)\dots f^{\lfloor\frac{n}{2}\rceil}(0)|^{\frac{1}{2+ o(1)}},\end{equation} 
which is violated for sufficiently large $n$. For example, if $c>0$, then $f^m(0)>f(0)\cdot f^2(0)\dots f^{m-1}(0)$ for all $m$. Our inequality then becomes  
\begin{equation}\label{Twist} |f^{n-2}(0)| \ll |d|^{\frac{1}{2+ o(1)}}\leq|2\cdot f(0)\cdot f^2(0)\dots f^{\lfloor\frac{n}{2}\rceil}(0)|^{\frac{1}{2+ o(1)}}\leq |2\cdot f^{\lfloor\frac{n}{2}\rceil}(0)|^{\frac{2}{2+o(1)}},\end{equation} 
 a more obviously violated relation. On the other hand, if $c\leq-3$, then we use the fact that $|f^m(0)|\geq(f^{m-1}(0)-1)^2$, as seen in \cite[Corollary 1.3]{Stoll-Galois}, to show that \[ |f(0)\cdot (f(0)-2)\cdot (f^2(0)-2)\dots\cdot (f^{n-3}(0)-2)|\leq |f^{n-2}(0)|\ll |d|^{\frac{1}{2+ o(1)}}\leq|2\cdot f(0)\cdot f^2(0)\dots f^{\lfloor\frac{n}{2}\rceil}(0)|^{\frac{1}{2+ o(1)}}.\] This is impossible for large enough $n$, and we conclude that $n$ is bounded as claimed.          
\end{proof} 
\begin{remark} 	Also assuming the ABC conjecture, Gratton, Nguyen, and Tucker have deduced the existence of primitive prime divisors in the critical orbit of $f_c$. This result, when coupled with the description of the primes that ramify in the quadratic extensions of $K_m$, also imply a finite index result; see \cite{PrimitiveDivisors} . 
\end{remark}

For a given $f$, as seen in the proof of Theorem $2$, an understanding of the hyperelliptic curves \[B_{n,f}:= \{(x,y): y^2=f_{\gamma,c}^n(x)\},\] and their quadratic twists coming from elements of the critical orbit, can provide information about the Galois behavior of $f$'s iterates. If $m\geq n$, then $B_{m,f}$ maps to $B_{n,f}$ via $(x,y)\rightarrow (f^{m-n}(x),y))$, and so we have a decomposition $\Jac(B_{m,f})\sim\Jac(B_{n,f})\times A_{m,n}$, for some abelian variety $A_{m,n}.$ Determining the factorization of $A_{m,n}$ into simple abelian varieties could be a step towards understanding the dynamical Galois groups ascociated to $f$. 

As an example, note that $B_{3,f}$ also maps to the curve  \[B_{3,2}:= \{(x,y): y^2=(x-c)\cdot f_{\gamma,c}^2(x)\},\] of genus $2$, via $(x,y)\rightarrow (f(x), (x-\gamma)\cdot y)$ (there are of course analogous maps at every stage). Moreover, we have similar maps for any quadratic twist of $B_{3,f}$. If $C^{(d)}$ denotes the quadratic twist of a hyperelliptic curve $C$ by $d\in\mathbb{Q}/(\mathbb{Q})^2$, then we conclude this section with an example that links the dynamics of quadratic polynomials with the ranks of the Jacobians of $B_{3,2}^{(d)}$, using the method of Chabauty and Coleman \cite{Poonen}. 

\textbf{Example 1}: Consider $f=x^2+3$ and the corresponding hyperelliptic curve \[B_{3,2}:=\{(x,y): y^2=(x-3)\cdot (x^4+6x+12)\},\] suppressing $f$,$\gamma$ and $c$ in the notation. We study this example in particular because the Galois groups of the iterates of $f$ are not known in this case. $B_{3,2}$ has good reduction at $p=5$, and $|B(\mathbb{F}_5)|=5$. Fix a twist $d=\prod_ip_i$, where the $p_i$'s are distinct primes dividing $2\cdot\prod_{j=1}^{n-1}f^j(0)$ and note that $B_{3,2}^{(d)}$ also has good reduction at $p=5$, since otherwise $5|d$ and $2$ is a quadratic residue in $\mathbb{F}_5$. Checking the possible values of $d\bmod{5}$ we conclude that $|B_{3,2}^{(d)}(\mathbb{F}_5)|\leq7$.

Suppose $d$ is such that $dy^2=f^n(0)$ has a rational solution $y$ for some $n$ (such a rational point forces the $n$-th iterate of $f$ to have smaller than expected Galois group). Then \cite[Theorem 5.3 (b)]{Poonen} implies that either $\rank(\Jac(B_{3,2}^{(d)})(\mathbb{Q}))\geq 2$ or the $\#\{n:dy^2=f^n(0)\}\leq4$, since each $n$ would give a point $(f^{n-2}(0), f^{n-3}(0)\cdot\pm{y})$ on $B_{3,2}^{(d)}$.

A detailed examination of these curves and their twists may very well provide a link between Galois phenomena and curves whose Jacobians have large rank. We will not address this here, but rather we begin a concrete study of a particular family of quadratic polynomials and a particular $n$, parametrizing their Galois behavior using the rational points on elliptic curves.   
 \end{section}
\begin{section}{Classification of Small Third Iterates when $\gamma=0$}
Throughout this section, we fix the critical point $\gamma=0$, hence study the family $f_c=x^2+c$, and suppress $\gamma$ in all notation. 

A particularly interesting example is $c=3$ and $n=3$. Although \[\Gal((x^2+3)^2+3)\cong D_4\cong\Aut(T_2),\] one computes that $|\Gal(((x^2+3)^2+3)^2+3)|=64$ , and hence $f_3$ has a small $3$rd iterate (generically $\Aut(T_n)$ is a wreath product of many copies of $\mathbb{Z}/2\mathbb{Z}$, an order $2^{2^n-1}$ Sylow $2$-subgroup of $S_{2^n}$). In fact, we will show $c=3$ is the only integer such that $f_c$ has a small third iterate. To prove this, consider the curves 
\begin{align*}
   E_1: -y^2&=t^3+2t^2+t+1,\\
   E_2: y^2(t+1)&= t^3+2t^2+t+1,\\
   V_3&=E_1\cup E_2\subset\mathbb{A}_{\mathbb{Q}}^2,
\end{align*}
where $V_3$ is the union of two elliptic curves (rather the union of their affine models). The following theorem provides an association of rational numbers with small $3$rd iterate and the rational points of $V_3$. As before, $\pi$ denotes the projection onto the first coordinate.  
\begin{theorem} There is an inclusion 
\begin{equation}
  \label{defS3}
  S_0^{(3)}:=\{c\in\mathbb{Q}| \Gal(f_c^3)\ncong\Aut(T_3),
  \Gal(f_c^2)\cong \Aut(T_2 ) \}\subset \pi(V_3(\mathbb{Q})).
\end{equation}
Moreover, the complement of $S_0^{(3)}$ in $\pi(V_3(\mathbb{Q}))$ is $\{0,-2\}$. In words, the rational  $t$-coordinates of $V_3(\mathbb{Q})$, excluding $t=-2$ and $t=0$, are in bijection with the rational numbers having small third iterate. 
\end{theorem}
\begin{proof}  
We first establish that if $c\in\mathbb{Q}^*$, then $f_c^3(0)\notin(\mathbb{Q})^2$. To do this, we study the rational points on the quartic curve, \[C_3: y^2=f_t^3(0)=(t^2+t)^2+t=t^4+2t^3+t^2+t.\] Then $C_3$ is birational (over $\mathbb{Q}$) to the elliptic curve given by the Weirstrass equation:
\[
  C_3': y^2= x^3+x^2+2x+1,\quad (t,y)\rightarrow (1/t,y/t^2).
\]
 Next we compute the rank of the rational points on $C_3'$. There are various ways to show $\rank(C_3'(\mathbb{Q}))=0$, either by employing a 3-descent \cite{Coh}, or computing its analytic rank (a theorem of Gross-Zagier and Kolyvagin implies the algebraic rank and analytic rank are equal in this case, given the modularity theorem of Wiles). Finally, by reducing our curve at various primes of good reduction (discriminant is $-2^4\cdot23$), we determine that the size of $C_3'(\mathbb{Q})_{Tor}=C_3'(\mathbb{Q})$ is three, and conclude $C_3'(\mathbb{Q})=\{(0,\pm{1}),\infty\}$. For a reference on computing various objects associated to elliptic curves, see either \cite{Coh} or \cite{Silv}. It follows that the only $c\in\mathbb{Q}$ with $f_c^3(0)$ a rational square is $c=0$.

Now for the proof of Theorem $2$, which will come in two parts: first showing that $S^{(3)}\subset \pi(V_3(\mathbb{Q}))$, followed by the computationally more difficult part, determining the complement of $S^{(3)}$. The latter task is equivalent to finding all rational points on certain higher genus curves.

 Suppose $c\in S^{(3)}$, then in particular $-c\notin\mathbb{Q}^2$ and $f_c^2(0)=c^2+c\notin\mathbb{Q}(\sqrt{-c})^2$. In fact this is a necessary and sufficient condition for $\Gal(f_c^2)\cong\Aut(T_2)$; see \cite{Stoll-Galois} or \cite{Jones1}. Let $K_{2,c}$ be the splitting field of $f_c^2$, by assumption a degree eight extension of the rationals. In this generic case, there is a simple description of the Galois group, namely $\Gal(K_{2,c}/\mathbb{Q})\cong\Aut(T_2)\cong D_4$. It follows that $K_{2,c}$ contains exactly three quadratic subfields: \[K_{2,c}^1=\mathbb{Q}(\sqrt{-c}),\hspace{.25in} K_{2,c}^2=\mathbb{Q}(\sqrt{c^2+c}), \hspace{.25in}K_{2,c}^3=\mathbb{Q}(\sqrt{-(c+1)}).\]
Let $K_{3,c}$ be the splitting field of the third iterate of $f_c$. Since the set $\{-c, f_c^2(0), f_c^3(0)\}$ does not contain a rational square, $f_c^3$ is necessarily irreducible (see Theorem $2.2$ in \cite{Jones1}). It follows that $\Gal(K_{3,c}/\mathbb{Q})\ncong\Aut(T_3)$ if and only if $f_c^3(0)=c^4+2c^3+c^2+c\in (K_{2,c})^2$ (Lemma 2.2 \cite{Jones2}).  

 Whence, $c\in S^{(3)}$ implies that there exists a $y\in K_{2,c}$ such that $y^2=f_c^3(0)$, i.e., $(c,y)\in C_3(K_{2,c})$. In particular $y\in\mathbb{Q}$ or $\mathbb{Q}(y)=K_{2,c}^i$ for some $1\leq i\leq 3$. Since the first case is impossible ($c\neq 0$ and opening remark), $\mathbb{Q}(y)=K_{2,c}^i$ for some $i$. Also $C_3(\mathbb{Q})=\{(0,0)\}$ implies $y=y'\sqrt{-c},y'\sqrt{c^2+c},$ or $y'\sqrt{-(c+1)}$ for some $y'\in\mathbb{Q}$.
 
For example, if $y=x'+y'\sqrt{-c}$ with $x',y'\in\mathbb{Q}$, then \[2x'y'\sqrt{-c}=(y')^2c-(x')^2+c^4+2c^3+c^2+c\in\mathbb{Q},\] so that either $x'$ or $y'$ is zero. But $y'$ cannot be zero, since this would yield a nontrivial rational point on $C_3$.   

Replacing $y$ with $y'$ and dividing by $c$ if convenient, we obtain a rational point on one of the following curves:\begin{align*}
  {E_1}: -y^2&=t^3+2t^2+t+1,\\
  {E_2}: y^2(t+1)&= t^3+2t^2+t+1,\\
  {C} : -y^2(t+1)&=t^4+2t^3+t^2+t.
\end{align*}
\begin{remark}
Using the group law on $E_1$ and $E_2$, we obtain an easy way to compute interesting rational numbers whose third iterate has a Galois group which is smaller than expected: $(-2/3, 25/9 ), (6/19, 515/361)\in E_2(\mathbb{Q})$ and $(-2,1), (-17/4,-53/8)\in E_1(\mathbb{Q})$, corresponding to $\Gal(f_c^3)$ of sizes $16, 64, 8$, and $64$ respectively.      
\end{remark}

We show that we need not consider rational points on $C$. 

\begin{lemma} $C(\mathbb{Q})=\{(0,0)\}$, and so we need not consider the case that $(c,y)$ is a point on $C_3(K_{2,c}^3)$.   
\end{lemma}
\begin{proof} $C$ maps birationally to the hyperelliptic curve $y^2= x(x-1)(x^3-2x^2+x-1)$ of genus two via $(t,y)\rightarrow(-t,y(t+1))$. Let $J$ be the Jacobian of the hyperelliptic curve. Following the $2$-descent procedure described in \cite{Stoll-de}, we determine that the $2$-Selmer group of $J$ is isomorphic to $\mathbb{Z}/2\mathbb{Z}\oplus\mathbb{Z}/2\mathbb{Z}$. Since we also have $J(\mathbb{Q})_{Tor}\cong\mathbb{Z}/2\mathbb{Z}\oplus\mathbb{Z}/2\mathbb{Z}$, it follows from the usual inequality, \[\rank(J(\mathbb{Q}))\leq\dimension_{\mathbb{F}_2}\Sel^{(2)}(\mathbb{Q}, J)- \dimension_{\mathbb{F}_2}J(\mathbb{Q})[2],\] that $\rank(J(\mathbb{Q}))=0$ (see Stoll's paper on implementing 2-descent for various computational heuristics, many of which have been made available in Magma \cite{Stoll-de}). In the rank zero case the method of Chabauty and Coleman can be used to prove that the only rational points on our hypereliptic curve are the Weierstrass points. In fact, in this special case, Magma can be used to prove it \cite{Magma}. Computing preimages of these points, we determine that $(0,0)$ is the only non-infinite rational point on $C$.  
\end{proof}
Returning to the proof of the Theorem $3$, we have seen that if $c\in S^{(3)}$, then we obtain a rational point $(c,y)$ on $V_3$. Now to determine the complement of $S^{(3)}\subset \pi_3(V_3(\mathbb{Q})).$ Fix a point $P=(c,y)\in V_3(\mathbb{Q})$ and assume $c\notin S^{(3)}.$ We can, without loss of generality, assume $c\neq 0$.

If $-c$, $-(c+1)$, and $c^2+c$ are not in $\mathbb{Q}^2$, then $c\in S^{(3)}$ (otherwise $\Gal(f_c^2)\cong\Aut(T_2)$ and $\Gal(f_c^3)\ncong\Aut(T_3)$, see \cite{KC} and \cite{Jones1}). Hence there are six cases to check, corresponding to whether $P\in E_1(\mathbb{Q})$ or $E_2(\mathbb{Q})$ and $-c,-(c+1)$, or $c^2+c$ are rational squares. We will see, applying Faltings' Theorem on the finiteness of rational points of curves with genus at least two \cite{Falt}, that there are only finitely many such $P$'s in each case. Explicitly finding these points remains a difficult task in general, but in our case all points can be found by some standard methods; see \cite{Flynn}, \cite{Stoll-R}.  

Consider first the case $P\in E_1$. 
 \par\textbf{(1.)} If $P=(c,y)\in E_1(\mathbb{Q})$ and $-c\in\mathbb{Q}^2$, write $-c=x^2$. Then $(x,y)$ is a rational point on the hyperelliptic curve $H:y^2= x^6-2x^4+x^2-1$, which has genus two. Fortunately, $H$ has no non-infinite rational points. To see this, we note that $H$ covers the elliptic curve $E':y^2=a^3+a^2+2a+1$ via the map $(x,y)\rightarrow (-1/x^2,y/x^3)$. $E'$ contains a rational subgroup $\{(0,\pm{1}),\infty\}$ of order three, and so with a little work, one can do a 3-descent on $E'$ by hand \cite{Coh}. Alternatively, a more standard 2-descent on $E'$ using Sage \cite{Sage}, shows that $(0,\pm{1})$ are the only non-infinite rational points on $E'$, which combined with the fact that $x=0$ does not give a rational point on~$H$, shows that $H$ has no non-infinite rational points. So we add no points to the complement of $S^{(3)}$ in this case. \par

\textbf{(2.)} Similarly, if $-(c+1)\in\mathbb{Q}^2$, say $-(c+1)=X^2$, then $(X,y)$ is a rational point on $\mathcal{C}:y^2=X^6+X^4-1$, again a curve of genus two. However, unlike the previous case, one must work harder to describe completely the rational points on this curve. The obstruction is that we can only say its Jacobian has rank at most $2$, precluding the possibility of applying the method of Chabauty and Coleman directly to $\mathcal{C}$ \cite{Poonen}. However, $C$ covers the elliptic curves $\mathcal{E}:y^2=x^3+x^2-1$ and $\mathcal{E}'$$:y^2=-x^3+x+1$, each of rank one, and we may apply a method of Flynn-Wetherell \cite{Flynn}, to compute the rational points on $C$.
\begin{lemma} $\mathcal{C}(\mathbb{Q})=\{(\pm{1},\pm{1}),\infty^{\pm}\}$. 
\end{lemma}
\begin{proof} We outline this approach here, conforming to the notation and conventions of \cite{Flynn}. A two descent on $\mathcal{E}$ yields $\mathcal{E}(\mathbb{Q})/2\mathcal{E}(\mathbb{Q})=\{\infty, (1,1)\}$, so that if $\alpha$ satisfies $\alpha^3+\alpha^2-1=0$, we study the two elliptic curves:
 \begin{align*}
 \mathcal{E}_1: y^2&= x(x^2+(\alpha+1)x+(\alpha^2+\alpha)),\\
 \mathcal{E}_2: y^2&=(1-\alpha)x(x^2+(\alpha+1)x+(\alpha^2+\alpha)),
 \end{align*}
defined over $\mathbb{Q}(\alpha)$. Suppose $(X,y)$ is a rational point on $\cal{C}$, from which it follows that $(x,y)=(X^2,y)$ is a rational point on $\cal{E}$. Taking the image of $(x,y)$ in $$\mathcal{E}(\mathbb{Q})/2\mathcal{E}(\mathbb{Q})\subset\mathbb{Q}(\alpha)^*/(\mathbb{Q}(\alpha)^*)^2,$$ and using the fact that $x\in\mathbb{Q}^2$, we find that $(x,y)\in\mathcal{E}_i(\mathbb{Q}(\alpha))$ for a unique choice of $i$ (see Lemma 1.1 in \cite{Flynn}). We now use the elliptic curve Chabauty method to find all points on $\mathcal{E}(\mathbb{Q}(\alpha))$ having rational $x$-coordinate, exploiting the formal groups of the curves over specified completions.
 
This more sophisticated method is not necessary for determining if $x$ gives a $\mathbb{Q}(\alpha)$-point on $\mathcal{E}_1$. A two descent on $\mathcal{E}_1$ yields \[\mathcal{E}_1(\mathbb{Q}(\alpha))\cong\mathbb{Z}/2\mathbb{Z}, \medspace\mathcal{E}_1(\mathbb{Q}(\alpha)=\{(0,0),\infty\}.\] However, $x=0$ does not give a rational point on $\mathcal{E}$, and so is irrelevant in determining $\mathcal{C}(\mathbb{Q})$.
   
 As for the second curve $\mathcal{E}_2$, the situation is more interesting. Magma, \cite{Magma}, computes that \[\Sel^{(2)}(\mathcal{E}_2,\mathbb{Q}(\alpha))\cong\mathbb{Z}/2\mathbb{Z}\oplus\mathbb{Z}/2\mathbb{Z}\quad\text{and}\quad\mathcal{E}_2(\mathbb{Q}(\alpha)_{Tor}\cong\mathbb{Z}/2\mathbb{Z},\] so that $\mathcal{E}_2(\mathbb{Q}(\alpha))\cong\mathbb{Z}/2\mathbb{Z}\oplus\mathbb{Z}$. In fact $P_0=(1,1)$ and $(0,0)$ generate the $\mathbb{Q}(\alpha)$-points on $\mathcal{E}_2$.
 
 \begin{remark} To show this fact, one can use bounds on the difference between canonical heights and the Weil height, followed by a point search \cite{SilvB}, or new lower bounds on the canonical heights over number fields \cite{T-T}.
 \end{remark}  
 
 Now, $\mathcal{E}_2$ has good reduction at $p=3$ (inert in $\mathbb{Q}(\alpha))$, and the reduction of $P_0$ has order three. Let $Q=3P_0$, so that every $P\in\mathcal{E}_2(\mathbb{Q}(\alpha))$, can be written as $S+nQ$, where $S\in\mathcal{S}=\{\infty, (0,0), \pm P_0, (0,0)\pm P_0\}.$ We see that $p=3$ satisfies the necessary criteria of 2.13 in \cite{Flynn}, and that if $x(P)\in\mathbb{Q}$, then one of the following three possibilities must be true: \[(a): P=nQ,\quad (b): P=(0,0)+nQ,\quad(c): P=nQ \pm P_0\]  
 for some $n$. We now use the formal group law on $\mathcal{E}_2(\mathbb{Q}_3(\alpha))$ to compute the $x$-coordinates of $nQ$, and hence of $P$, determining which rational $x$-coordinates are permitted (the key idea being that $\rank(\mathcal{E}_2(\mathbb{Q}(\alpha)))<|\mathbb{Q}(\alpha):\mathbb{Q}|$, a Chaubauty-like condition). 
 \begin{remark}This can actually be made into exactly a Chaubauty condition using Weil's restriction of scalars functor. Let $E/K$ be an elliptic curve and let $L/K$ be a finite extension. Then there is an abelian variety $A/K$, called the $L/K$-restriction of scalars of~$E/K$, having the following properties: (1) $A(K)\cong E(L)$. (2) $\dim A=[L:K]$. So in particular, if $\rank E(L)<[L:K]$, then $A/K$ exactly satisfies the Chaubauty condition $\rank A(K)<\dim A$.
\end{remark}  
 Let us recall the relevant power series \cite{Flynn}. If  $y^2 = g_3x^3 + g_2x^2 + g_1x$ is an elliptic curve, and $z=\frac{-x}{y}$, then \[1/x=g_3z^2
+ g_2g_3z^4+ (g_1g_3^2+g_3g_2^2)z^6+ (g_2^3g_3+3g_1g_3^2g_2)z^8 +O(z^{10}).\]
Moreover, if $(x_0,y_0)$ is not in the kernel of reduction, and $(x_0,y_0)$ + $(z/w,-1/w) = (x_3,y_3)$, then:  
\begin{align*}
{x_3= x_0+ 2y_0z + (3x_0^2g_3+2x_0g_2+g_1)z^2+ (4x_0g_3y_0+2g_2y_0)+z^3+}\\{(4x_0^3g_3^2+6x_0^2g_3g_2+2x_0g_3g_1+2x_0g_2^2+g_3y_0^2+g_2g_1)z^4+ O(z^6)}
\end{align*}
and 
\begin{align*}{\log(z)} := z+\frac{1}{3}g_2z^3+(\frac{1}{5}g_2^2+\frac{2}{5}g_1g_3)z^5+(\frac{1}{7}g_2^3+\frac{3}{7}g_0g_3^2+\frac{6}{7}g_2g_1g_3)z^7+O(z^9),\\
{\exp(z)} := z-\frac{1}{3}g_2z^3+(\frac{-2}{5}g_1g_3+\frac{2}{15}g_2^2)z^5+(\frac{-17}{315}g_2^3+\frac{22}{105}g_2g_1g_3)z^7+O(z^9).\end{align*}
Case(a) : If $P=(x,y)=n*Q$ and $z=-\frac{-x}{y}$, then set $(x_n,y_n)= n*Q$ and $z_n=\frac{-x_n}{y_n}$. We use $z_n=\exp(n\log(z))$, and group terms to write \[ \frac{1}{x_n}= \phi_0+\phi_1\alpha+\phi_2\alpha^2,\quad\phi_i\in\mathbb{Z}_3[[n]].\]  
Now, \[z_n=n\log(z)-\frac{1}{3}(1-\alpha^2)n^3(\log(z))^3+\frac{1}{15}(2(1-\alpha^2)^2-6(1-\alpha)(1-\alpha^2)(\alpha^2+\alpha))n^5\log(z)^5+\dots,\] and it will suffice to work $\pmod{3^4}$. One computes that \[z\equiv3(5\alpha^2+20\alpha+9)\pmod{3^4}, \,\text{and}\; z_n\equiv(15\alpha^2+60\alpha+18)n+72n^3\pmod{3^4}.\] It follows that 
\begin{align*}
\frac{1}{x_n}&\equiv(72\alpha^2+63)n^2+(54\alpha^2+54\alpha+27)n^4\\
&\equiv(63n^2+27n^4)+(54n^4)\alpha +(72n^2+54n^4)\alpha^2\pmod{3^4}.
\end{align*}
Hence, $\frac{1}{x_n}\in\mathbb{Q}$ implies $\phi_2=O(n^2)\in \mathbb{Z}_3[[n]]$, and $\phi_2\equiv72n^2+54n^4\bmod{3^4}$. Note that $\phi_2$ has a double root at $0$, and $|72|_3=3^{-2}$, which is strictly larger than the $5$-adic norm of all subsequent coefficients. Applying a theorem of Strassmann's, we see that there are no other $n\in\mathbb{Z}_3$ (and so no other $n\in\mathbb{Z})$ satisfying $\phi_2(n)=0$; see \cite{Flynn}. 

For Cases (b) and (c), replace $(x_0,y_0)$ by $(0,0)$ in case (b) and  $(x_0,y_0)$ by $(1,1)$ in case (c) in the addition formula above. As in case (a), work $\pmod{3^4}$ and use the same theorem of Strassmann's to deduce $n=0$. We conclude that if $P=(x,y)$ is such that $P\in \mathcal{E}_i(\mathbb{Q}(\alpha))$ and $x\in\mathbb{Q}$, then $x=0,1$. It follows that $X=\pm{1}$ as claimed.               
 \end{proof}
 \begin{remark} The solutions $X=\pm{1}$ correspond to $D=-2$. Indeed, $-2$ is not in $S^{(3)}$.  
 \end{remark} 

\textbf{(3.)} Suppose $c^2+c$ is a rational square. The conic $c^2+c=x^2$ is rational, with point $(0,0)$. It follows that such $c$ are parametrized by $c=\frac{1}{a^2-1}$, for $a\in\mathbb{Q}$. Then by sending $(t,y)\rightarrow (1/t,y/t^2)$, we obtain rational points on $v^2=-w(w^3+w^2+2w+1)$. We make the change of variables, $a^2-1=\frac{1}{c}=w$, so that $(a,v)$ is a rational point on \[\mathcal{C'}:v^2=-(a-1)(a+1)(a^6-2a^4+3a^2-1),\] a curve of genus three. We study unramified covers of $\mathcal{C'}$ to find all rational points.
\begin{lemma} $\mathcal{C'}=\{\infty^{\pm}, (\pm{1},0)\}$
\end{lemma} 
\begin{proof}
 Let \[\mathcal{D}:u^2=-(a^2-1),\quad s^2=a^6-2a^4+3a^2-1,\] then $\mathcal{D}$ is an unramified $\mathbb{Z}/2\mathbb{Z}$-covering of $\mathcal{C'}$, with map $\pi:\mathcal{D}\rightarrow\mathcal{C'}$ given by $(a,u,s)\rightarrow(a,us)$. The twists of $\mathcal{D}$ are \[\mathcal{D}_d: du^2=-(a^2-1),\quad ds^2=a^6-2a^4+3a^2-1,\quad d\in\mathbb{Q}^*/(\mathbb{Q}^*)^2,\] and every rational point on $\mathcal{C'}$ lifts to a rational point on one of these twists \cite{Stoll-R}. If $p$ is a prime divisor of $d$, then $\mathcal{D}_d$ has no $p$-adic points (hence no rational points) unless both $-(a^2-1)$ and $a^6-2a^4+3a^2-1$ have common roots modulo $p$ \cite{Stoll-R}. Since the resultant of $-(a^2-1)$ and $a^6-2a^4+3a^2-1$ is $1$, we need only consider $\mathcal{D}$ and its twist $\mathcal{D}_{-1}$. However, we can easily describe both $\mathcal{D}(\mathbb{Q})$ and $\mathcal{D}_{-1}(\mathbb{Q})$: \[\mathcal{D}(\mathbb{Q})=\{(\pm{1},0,\pm{1})\} ,\quad \mathcal{D}_{-1}(\mathbb{Q})=\emptyset.\] To see this, note that the equation $-s^2=a^6-2a^4+3a^2-1$ in the definition of $\mathcal{D}_{-1}$ covers the elliptic curve $E: b^2=z^3+2z^2+3z+1$, and a $2$-descent shows that $E(\mathbb{Q})=\{(0,\pm1),\infty\}$. This forces the equality $-u^2=-(0^2-1)$, which has no rational solutions. Similarly for $\mathcal{D}$, the equation $s^2=a^6-2a^4+3a^2-1$ covers the elliptic curve $E': b^2=z^3-2z^2+3z-1$, and a $2$-descent yields $E'(\mathbb{Q})=\{(1,\pm{1}),\infty\}$. Putting these statements together we deduce that $\mathcal{C'}=\{\infty^{\pm}, (\pm{1},0)\}$ as claimed. 
 \end{proof}
 Note that $a=\pm{1}$ implies $c$ is infinite, and so we add no rational $(c,y)$ on $E_1$ with $c\notin S^{(3)}$ in this case.   

We now summarize the computations arising in the $E_2$ cases below: 

\textbf{(1)}. For the points $P=(c,y)\in E_2(\mathbb{Q})$ with either $-c$, $c^2+c$, or $-(c+1)$ a rational square, we use a similar approach. If $-(c+1)=x^2$, then $x$ satisfies $v^2=x^6+x^4-1$ for a rational $v$. Notice that in case two for $E_1$ above, we have already determined the rational points on this curve: $(\pm{1},\pm{1})$, corresponding to $c=-2$ as before.

\textbf{(2)}. Similarly, if $P\in E_2(\mathbb{Q})$ and $c^2+c=x^2$, then $c$ is parametrized by $\frac{1}{a^2-1}=c$. One makes a change of variables to find that $a$ lies on the curve \[\mathcal{A}:v^2=a^6-2a^4+3a^2-1,\] for a rational $v$. The hyperelliptic curve $\mathcal{A}$ covers the elliptic curve $F:c^2=b^3-2b^2+3b-1$, via $(a,v)\rightarrow(a^2,v)$. A standard two descent and torsion algorithm show that $\rank(F(\mathbb{Q}))=0$, and $F(\mathbb{Q})=\{(1,\pm{1}),\infty\}.$ It follows that \[\mathcal{A}(\mathbb{Q})=\{(\pm{1},\pm{1}),\infty^{\pm}\},\] yet again yielding no finite $c$'s.  

\textbf{(3)}. The last case, $-c=x^2$, does require more care. In this case, one finds that $x$ lies on the curve: \[\mathcal{B}: y^2=(1-x^2)(-x^6+2x^4-x^2+1),\] a curve of genus three. We again use unramified covers of $\mathcal{B}$ to reduce the problem. 
\begin{lemma} $\mathcal{B}(\mathbb{Q})=\{(0,\pm{1}), (\pm{1},0), \infty^{\pm}\}$. 
\end{lemma}
\begin{proof}
The relevant covers are \[\mathcal{D}:u^2=1-x^2,\quad s^2=-x^6+2x^4-x^2+1,\] and its twist \[\mathcal{D}_d: du^2=1-x^2,\quad ds^2=s^2=-x^6+2x^4-x^2+1,\quad d\in\mathbb{Q}^*/(\mathbb{Q}^*)^2,\]  (see case 3 above). Again the resultant of $1-x^2$ and $-x^6+2x^4-x^2+1$ is $1$, and so we need only compute $\mathcal{D}(\mathbb{Q})$ and $\mathcal{D}_{-1}(\mathbb{Q})$ as before. We will show that $\mathcal{D}_{-1}(\mathbb{Q})=\emptyset$. To see this, note that the second defining equation of $\mathcal{D}_{-1}$ is $s^2=x^6-2x^4+x^2-1$, which covers the elliptic curve $c^2=b^3+b^2+2b+1$, having only $b=0$ as a possible solution (2-descent). This leaves only the infinite points on $\mathcal{D}_{-1}$, which we disregard. 

It therefore suffices to find all rational points on $\mathcal{D}$ to recover the points on $\mathcal{B}$ \cite{Stoll-R}. To do this we find all rational points on \[\mathcal{U}: s^2=-x^6+2x^4-x^2+1.\] Unfortunately, the rational points of the Jacobian of $\mathcal{U}$ have rank two, and so we must apply the elliptic Chabauty method to $\mathcal{U}$, as in case $2$ of $E_1$, to describe the rational points on $\mathcal{B}$; see \cite{Flynn}. It follows that only $c=0$ is added to the complement of $S^{(3)}$ in this case.          
\end{proof}

In summary, we have shown that for all but $c\in\{-2, 0\}$, if $(c,y)\in V_3(\mathbb{Q})$ then $\Gal(f_c^2)\cong\Aut(T_2)$. It remains to show $\Gal(f_c^3)\ncong\Aut(T_3)$, to ensure that $c$ has a small third iterate. This is easy. Since $f_c^3(0)\in (K_{2,c})^2$, it suffices to show that $f_c^3$ is irreducible. But it is, as $\{-c,f_c^2(0), f_c^3(0)\}$ are all not rational squares (Theorem $2.2$ in \cite{Jones1} and Lemma 2.2 \cite{Jones2}). This completes the proof of the theorem.  
\end{proof}
As an immediate application of the correspondence described in our theorem, we determine completely the integers having small third iterate, answering a question of Rafe Jones. This is made possible by David's theorem on bounds in elliptic logarithms. 
\begin{corollary} $S^{(3)}\cap\mathbb{Z}=\{3\}.$ That is, $3$ is the only integer with small third iterate. 
\begin{proof}
 Applying the theorem above, it suffices to list the points with integral $t$-coordinates on $E_1(\mathbb{Q})$ and $E_2(\mathbb{Q})$. We will show that the only such $c\in\mathbb{Z}$ are $(-2,\pm{1})\in E_1$ and $(3,7/2), (0,\pm{1}), (-2,\pm{1})$ on $E_2$. As $c=0,-2$ are not in $S^{(3)}$ ($-2$ has a small second iterate), we will have proved $S^{(3)}\cap\mathbb{Z}=\{3\}$ as claimed. 

Suppose $(c,y)\in V_3(\mathbb{Q})$ and $c$ is an integer. We consider the case $(c,y)\in E_2(\mathbb{Q})$ here. Then, via the rational map $(t,y)\rightarrow (t,y(t+1))$, we have a point on the elliptic curve $y^2=(x+1)(x^3+2x^2+x+1)$ that has both integral coordinates. This curve has a rational point $(0,1)$, so we may apply the birational map \[(x,y)\rightarrow (2(x^2-y)+3x, x(4x^2+6x-4y+3/2)\] to obtain a rational point on the Weierstrass equation $y^2+3xy+4y=x^3+3/4x^2-4x-3$, having integral $x$-coordinate (see chapter $7$ of \cite{Coh} for methods of transforming various curves into standard form). Finally, by sending $(x,y)\rightarrow(x+1,y+3/2x+2)$, we get an integral point on the minimal Weierstrass equation $E: y^2=x^3-x+1$. We now use David's theorem on elliptic logarithms and the LLL algorithm to list all integral points on $E$; see \cite[\S8.8]{Coh}. However, there are other ways of finding the integral points on genus 1 equations, not necessarily in standard form; see \cite{Stroeker}.

A standard $2$-descent shows $E(\mathbb{Q})$ has rank one and no torsion. Using Silverman's bounds on the difference between the logarithmic and canonical heights \cite{SilvA}, we see that $P=(1,1)$ is a generator for $E(\mathbb{Q})$.  Furthermore $\disc(E)=-2^4\cdot23<0$, so that $E(\mathbb{R})$ is connected, and we need not worry about certain subtleties that arise when this is not the case. We keep track of various constants, in keeping with the notation used in \cite[\S8.8]{Coh} : \[h(E)=8.841, \mu(E)=2.9356, c_1=160.07, c_2=0.099617, c_3=8, c_5=35.785, c_7=.555,\] \[c_8=24.032, c_9=3.962,  \omega_1=4.767, \psi(P)=3.676, h_m(P)=8.841, n=2, c_{10}=4.074\cdot10^{40}.\] We deduce from David's Theorem and \cite[Corollary 8.73]{Coh}, that if $Q=N*P$ and $Q\in E(\mathbb{Z})$ then, $$-\log(\psi(Q))\geq(0.049805)N^2-3.578,$$ $$-\log(\psi(Q))\leq 4.074\cdot10^{40}\cdot(\log(N)+1.377)\cdot(\log(\log(N))+8.841+1.377)^2.$$
Such a system of inequalities is violated when $N>10^{25}$, hence $N\leq10^{25}$. At first glance, such an astronomical bound seems useless. However, the mere existence of a bound allows us to employ reduction techniques via the LLL-algorithm. In practice we can reduce $N$ substantially ($N<100$ for many curves with moderately sized coefficients). We input $C>10^{50}$, say $C=10^{60}$ to Proposition $2.3.20$ in \cite{Coh} with matrix $X$, and it's LLL-reduced matrix $Y$: 
\begin{align*}
  X&= \left( \begin{array}{cc}
      1 & 0 \\
      \lfloor C\psi(P)\rceil & \lfloor C\omega\rceil \\
    \end{array} \right),  \\
  Y&= \left( \begin{array}{cc}
    -928309378069697515001621255593 & -346795556312856677461017188696\\
    0 & -5070602400912917605986812821504\\ \end{array} \right) 
\end{align*} 
where, $\lfloor x \rceil$ denotes the nearest integer to a real number $x$. This can be done using MAGMA, \cite{Magma}. Then for $m\leq 10^{60}$, $$|m\omega+N\psi(P)|\geq\frac{\sqrt{8.4\times10^{59}-2\times10^{50}}-(\frac{1}{2}+10^{25})}{10^{60}},$$ and hence $$9\times10^{-31}\leq|m\omega+N\psi(P)|\leq35.785e^{-.049805N^2}.$$
These inequalities are incompatible for $N>40$, so we can conclude that $N\leq40$. We could run our reduction techniques again to further reduce $N$, however this is not necessary. Searching the forty points $\{N*P\}_{N\leq40}$ in Sage \cite{Sage}, we deduce that \[ \{(t,y)\in E(\mathbb{Q}):t\in\mathbb{Z}\}=\{(0,\pm{1}),(\pm{1},\pm{1}), (3,\pm{5}), (5,\pm{11}), (56,\pm{419})\},\]
and by retracing these point to $E_2$, conclude that \[\{t\in\mathbb{Z}: (t,y)\in E_2(\mathbb{Q})\}=\{-2,3\}.\]  
However, $|\Gal(f_{-2}^2)|=4$, so that $-2$ is not in $S^{(3)}$. 

Similarly, $E_1(\mathbb{R})$ is connected, $|E_1(\mathbb{Q})_{Tor}|=1$, and $\rank(E_1(\mathbb{Q}))=1$ with generator $P=(-2,1)$, and we can compute $E_1(\mathbb{Z})=\{(-2,1)\}$. The corollary follows.

It is worth noting that if one finds generators for the Mordell-Weil group, Sage has implemented a package which computes the integer points of a Weierstrass equation \cite{Sage}. Since our equations have small coefficients and rank one, it is easy to find a generator, and run this package. All results were checked in this way.  
\end{proof}
\end{corollary}
\end{section} 

\begin{section}{Elliptic Surface Parametrized by $\gamma$ and Higher Iterates} 
In order to study the Galois theory of third iterates for general quadratic polynomials, we view $\gamma$ as an indeterminate and begin our study with the surface $C_{3,\gamma}: y^2=f_{\gamma,c}^3(\gamma)$. Fortunately, for every $\gamma$ we have the rational point $(0,f_{\gamma,0}(0))$, hence an elliptic surface parametrized by $\gamma$. For $\gamma\neq0,1$ we see that $C_{3,\gamma}$ is birational to 
\[\mathcal{E}_{\gamma}: y^2 = x^3+a_2x^2+a_4x+a_6\]
with 
\begin{align*}
a_2 &= \frac{144}{13}\gamma^2 - \frac{147}{13}\gamma + \frac{67}{52},\\
a_4 &=  \frac{6912}{169}\gamma^4 - \frac{14112}{169}\gamma^3 + \frac{8811}{169}\gamma^2 - \frac{4635}{338}\gamma + \frac{6003}{2704},\\
a_6 &= \frac{110592}{2197}\gamma^6 - \frac{338688}{2197}\gamma^5 + \frac{384336}{2197}\gamma^4 - \frac{228889}{2197}\gamma^3 + \frac{365399}{8788}\gamma^2 - \frac{307667}{35152}\gamma + \frac{169073}{140608}.\\
 \end{align*}
The surface $\mathcal{E}\rightarrow\mathbb{P}^1$ has a section $(-2\gamma^2+2\gamma, 0)$ of infinite order. Moreover, the Weierstrass equation for the generic fiber $\mathcal{E}_\gamma$ is minimal over $\mathbb{Q}[\gamma]$, and so we may run Tate's algorithm to compute the local information at the bad fibers of $\mathcal{E}$, summarized in the following table \cite{SilvA}:
\begin{center}
\begin{tabular*}{0.75\textwidth}{@{\extracolsep{\fill}} | c | c | c | r | }
    \hline
 \quad\quad{$\mathcal{P}$, Place in $\mathbb{Q}(\gamma)$}& $v_{\mathcal{P}}(\Delta)$& Tamagawa \# & Kodaira Symbol \\ \hline
    $\frac{1}{\gamma}$ & $9$ & $2$ & \Rmnum{3}$^*$ \quad\quad\quad\\ \hline
   \quad$\gamma^3 - \frac{23}{16}\gamma^2 + \frac{13}{32}\gamma - \frac{23}{256}$ & $1$ & $1$ & \Rmnum{1}$_1$ \quad\quad\quad\\ \hline \end{tabular*}
\end{center}     
 However, $\mathcal{E}$ is a rational surface since $\deg(a_i)\leq i$ and the minimal projective model of $\mathcal{E}$ has a singular fiber \cite{SS}. It follows from the Shioda-Tate formula that $\rank(\mathcal{E}(\bar{\mathbb{Q}}(\gamma))=1$, and hence $\rank(\mathcal{E}({\mathbb{Q}}(\gamma))=1$, \cite{OS}. Then Silverman's specialization theorem \cite{SilvA} implies $\rank(\mathcal{E}_\gamma(\mathbb{Q}))\geq 1$ for all but finitely many $\gamma\in\mathbb{Q}$. We can thus conclude the following theorem:  

\begin{theorem} $S_\gamma^{(3)}$ is infinite for all but finitely many $\gamma\in\mathbb{Q}$. 
\end{theorem} 
\begin{proof} As remarked above, for all but finitely many $\gamma\in\mathbb{Q}$, we have that $\rank(\mathcal{E}_\gamma(\mathbb{Q}))\geq 1$. For such $\gamma$, the set $\{c\in\mathbb{Q} | f_{\gamma,c}(\gamma)\in\mathbb{Q}^2\}$ is infinite. However, the subset of $c$'s such that either $-f_{\gamma,c}(\gamma), f_{\gamma,c}^2(\gamma)$, or $-f_{\gamma,c}(\gamma)\cdot f_{\gamma,c}^2(\gamma)$ is a rational square is finite (each case gives a  rational point on a genus two or higher curve). In any event, for all but finitely many $c$'s, $\Aut(T_2)\cong \Gal(f_{\gamma,c}^2)$. If $f_{\gamma,c}^3$ is reducible, then certainly $c\in S_{\gamma}^{(3)}$. On the other hand, if $f_{\gamma,c}^3$ is irreducible, then $\Aut(T_3)\ncong \Gal(f_{\gamma,c}^3)$, since $f_{\gamma,c}^3(\gamma)\in(K_{2,c})^2$. The result follows from Lemma 3.2 in \cite{Jones2}.            
 \end{proof}
 
  There is however, noticeable sensitivity on $\gamma$, as evidenced by the following proposition (compare to the $\gamma=0$ case).     
\begin{proposition} There is an inclusion \[S_{1}^{(3)}\subset \{t\in\mathbb{Q}|\, (t,y)\;\;\text{satisfies}\, \;E: y^2=t^4-2t^3+t^2+t\;\; \text{for some}\; y\in\mathbb{Q}\},\] and E is a rank one elliptic curve. Moreover, the complement of $S_{1}^{(3)}$ is finite, and $S_1^{(3)}\cap\mathbb{Z}=\emptyset$.  
\end{proposition} 
\begin{proof} We work with $f_c(x) = (x-1)^2+c$. After normalization, the relevant curves are: 
\begin{align*}
{E: y^2=t^4-2t^3+t^2+t,\quad C_1: y^2= t^6 - 3t^5 + 4t^4 - 2t^3 + t},\\ 
{C_2: y^2=-t^3+2t^2-t-1, \quad C_3: y^2=-t^5+3t^4-4t^3+2t^2-1,}
\end{align*}
corresponding to $f_c^3(1)$ and the three quadratic subfields $\mathbb{Q}(\sqrt{(c-1)^2+c})$, $\mathbb{Q}(\sqrt{-c})$ and $\mathbb{Q}(\sqrt{ -c\cdot((c-1)^2+c)})$ of $K_{2,c}$ respectively.

The elliptic curve $C_2$ has rank zero with no rational points. The two genus two curves, $C_1$ and $C_3$, have Jacobians of rank one. Running the Chabauty function in Magma, we obtain \[C_1(\mathbb{Q})=\{(-1,\pm{3}),(0,0), (1,\pm{1}), \infty^{\pm}\},\quad C_3(\mathbb{Q})=\{(-1,\pm{3}),\infty\}.\]      
However $c=1$ is not in $S_1^{(3)}$, since its second iterate has Galois group of size four. Similarly, $c=-1,0$ are not in $S_1^{(3)}$, as they correspond to reducible polynomials. We conclude that $E$ is the only relevant curve, proving the first part of the statement. 

For the integer points, note that $E$ is birational to the elliptic curve $E': y^2-2xy+2y=x^3$, via the transformation $(t,y)\rightarrow (2(t^2-y)-2t, t(4t^2-4t-4y))$, see chapter $7$ of \cite{Coh}. It suffices therefore, to compute the points on $E'$ having integral $x$-coordinate. For this we use Sage and then compute preimages in $E$, finding only $t=1$. Because $c=1$ is not in $S_1^{(3)},$ the result follows.   
\end{proof}   
       
\begin{remark}As for larger iterates in the $\gamma=0$ case, we note that the curve $C_4: y^2=f_c^4(0)$ is of genus $3$, and its Jacobian, $J_4$, has rank zero! It is straightforward, therefore, to conclude that $C_4(\mathbb{Q})=\{\infty^{\pm}, (0,0),(-1,0)\}$. With this description, it may be possible to use methods as in the third iterate case.
\end{remark} 

 \textbf{Acknowledgments} It is a pleasure to thank my advisor Joe Silverman, as well as Rafe Jones, for their many helpful comments and valuable insight. I would also like to thank Bjorn Poonen for suggesting the use of unramified covers to determine the rational points on the higher genus curves within this paper.  
\end{section} 

\end{document}